\documentclass[12pt]{article}
\usepackage{amsmath,amssymb,amsthm}
\usepackage{graphicx,psfrag,float,subcaption,caption}
\usepackage{enumerate,pifont,mathrsfs}
\usepackage[normalem]{ulem}
\usepackage[scale=0.8,a4paper]{geometry}
\usepackage[colorlinks,linkcolor=blue,citecolor=blue]{hyperref}
\usepackage{tikz}

\newtheorem{theorem}{Theorem}[section]

\newtheorem{lemma}[theorem]{Lemma}

\newtheorem{conjecture}[theorem]{Conjecture}

\title{Covering Bridgeless Claw-Free Cubic Graphs with Perfect Matchings}
\author{
    {\small\bf Huaibin Zhou, \quad Xiumei Wang\thanks{Corresponding author. E-mail: wangxiumei@zzu.edu.cn}}\\
    {\small   School of Mathematics and Statistics, Zhengzhou University,}
    {\small Zhengzhou, Henan 450001, China}
}
\date{}

\begin{document}
\maketitle

\begin{abstract}
The Berge conjecture [Proc. London Math. Soc., 1979] asserts that every bridgeless cubic graph can be covered by five perfect matchings.
The 5-cycle double cover conjecture, proposed independently by Preissmann (1981) and Celmins (1984), asserts that every bridgeless graph admits five even subgraphs such that every edge belongs to exactly two of them.
Hakobyan and Mkrtchyan [Ars Math. Contemp., 2019] proved that the 5-cycle double cover conjecture holds if and only if every bridgeless claw-free cubic graph can be covered by four perfect matchings.
In this paper, we prove that the Berge conjecture holds for every bridgeless claw-free cubic graph $G$ and that $G$ has four perfect matchings covering at least $\frac{299}{315}|E(G)|$ edges.


\end{abstract}

\section{Introduction}

In this paper, all graphs are finite, connected, and loopless, but parallel edges are allowed. A graph is \emph{simple} if it has no parallel edges. We use the standard terminology and notation of~\cite{BondyMurty2008}. For a graph $G$, let $V(G)$ and $E(G)$ denote its vertex set and edge set, respectively.
A \emph{matching} of $G$ is a set of pairwise nonadjacent edges of $G$. A vertex $v$ or an edge $e$ is said to be covered by a matching $M$ if $v$ is incident with an edge of $M$ or $e\in M$, respectively. A matching of $G$ is \emph{perfect}
if it covers all vertices of $G$.
A graph $G$ is said to be covered by a family of perfect matchings of $G$ if every edge of $G$ is covered by at least one of these perfect matchings. 
Berge and Fulkerson proposed the following two conjectures, which were proved to be equivalent proved by  Mazzuoccolo~\cite{Mazzuoccolo2011}.

\begin{conjecture}[Berge-Fulkerson~\cite{Fulkerson1971}]\label{conj:bf}
	Every bridgeless cubic graph has six perfect matchings such that every edge belongs to exactly two of them.
\end{conjecture}

\begin{conjecture}[Berge~\cite{Seymour1979}]\label{conj:berge}
	Every bridgeless cubic graph can be covered by five perfect matchings.
\end{conjecture}

A \emph{claw} is a graph isomorphic to $K_{1,3}$, and a graph is \emph{claw-free} if it contains no induced claw. Our first main result is the following.

\begin{theorem}\label{thm:five-cover}
	Every bridgeless claw-free cubic graph can be covered by five perfect matchings.
\end{theorem}

A subgraph is \emph{even} if every vertex has even degree. A \emph{cycle double cover} of a graph is a family of even subgraphs such that every edge belongs to exactly two of them. A \emph{$k$-cycle double cover} is a cycle double cover consisting of $k$ even subgraphs. The cycle double cover conjecture, independently proposed by Szekeres and Seymour~\cite{BondyMurty2008}, asserts that every bridgeless graph has a cycle double cover. A proof of the cycle double cover conjecture was recently given~\cite{Oum2026}. The proof yields an 8-cycle double cover, while the stronger 5-cycle double cover conjecture remains open.

\begin{conjecture}[5-cycle double cover~\cite{Celmins1984,Preissmann1981}]\label{conj:5cdc}
	Every bridgeless graph has five even subgraphs such that every edge belongs to exactly two of them.
\end{conjecture}

Several cases of the 5-cycle double cover conjecture are known. Huck and Kochol~\cite{HuckKochol1995} proved it for bridgeless cubic graphs of oddness at most two, and Huck~\cite{Huck2001} extended this to oddness at most four.
Zhang~\cite{Zhang1996} and, more recently,  Li et al.~\cite{LiHaoLuoZhang2025} and Liu et al.~\cite{LiuHaoLuoZhang2023,LiuHaoLuoZhang2025}  obtained several sufficient conditions for the existence of 5-cycle double covers in terms of nowhere-zero $4$-flows, Catlin reductions, shortest cycle covers, and non-separating cycles.
The following equivalence, due to Mazzuoccolo, appears in the paper of Hakobyan and Mkrtchyan~\cite{HakobyanMkrtchyan2019}.

\begin{theorem}[\cite{HakobyanMkrtchyan2019}]\label{thm:HM}
	Conjecture~\ref{conj:5cdc} holds if and only if every bridgeless claw-free cubic graph can be covered by four perfect matchings.
\end{theorem}

It remains open whether every bridgeless claw-free cubic graph can be covered by four perfect matchings. In general, four perfect matchings do not always cover all edges of a bridgeless cubic graph, and several infinite families of such graphs are known~\cite{EsperetMazzuoccolo2014,LukotkaOguagbaka2025,MacajovaSkoviera2021,SkovieraVarsa2024}.
In this paper we study how many edges can be covered by four perfect matchings in bridgeless claw-free cubic graphs.

Let $m_4$ denote the largest constant such that every bridgeless cubic graph $G$ has four perfect matchings covering at least $m_4|E(G)|$ edges. The Petersen graph gives the upper bound $m_4\leq\frac{14}{15}$, and Patel~\cite{Patel2006} conjectured that $m_4=\frac{14}{15}$, which follows from Conjecture~\ref{conj:bf}. For simple bridgeless cubic graphs, Mazzuoccolo~\cite{Mazzuoccolo2013} proved that $m_4\geq\frac{55}{63}$.

Our second main result is the following, which implies that if $m_4=\frac{14}{15}$, then every extremal graph for $m_4$ must contain a claw, as $\frac{299}{315}>\frac{14}{15}$.

\begin{theorem}\label{thm:intro-main}
	Every bridgeless claw-free cubic graph $G$ admits four perfect matchings that cover at least $\frac{299}{315}|E(G)|$ edges.
\end{theorem}

This paper is organized as follows. In Section~2, we present some preliminaries. Sections~3 and~4 are devoted to the proofs of Theorems~\ref{thm:five-cover} and~\ref{thm:intro-main}, respectively.

\section{Preliminaries}\label{sec:construction}

A \emph{diamond} is an induced subgraph isomorphic to $K_4-e$, see $Q_1$ in Figure~\ref{fig:diamonds}(a).
In a diamond, the edge  with ends of degree three is called the \emph{central edge}, and the other four edges are called the \emph{side edges}.
A \emph{2-cycle} is a circuit of length two, see $Q_2$ in Figure~\ref{fig:diamonds}(a).
A \emph{D2C-string} is a sequence $Q_1,\ldots,Q_k$ of diamonds or 2-cycles in which consecutive members are joined by an edge, see Figure~\ref{fig:diamonds}(a).
A \emph{D2C-ring} is obtained from a D2C-string by adding an edge joining two vertices of degree two, see Figure~\ref{fig:diamonds}(b). 

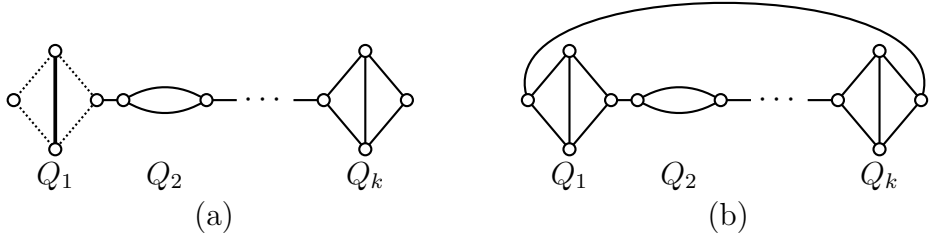
\begin{figure}[ht]
	\centering
	\begin{tikzpicture}[
		scale=1.0,
		line width=0.8pt,
		point/.style={
			circle,
			draw,
			fill=white,
			inner sep=1.6pt
		},
		central/.style={
			line width=1.4pt
		},
		side/.style={
			densely dotted
		}
		]

		\begin{scope}[shift={(-3.4,0)}]

			\coordinate (aLone) at (-2.65,0);
			\coordinate (aRone) at (-1.55,0);
			\coordinate (aTone) at (-2.10,0.65);
			\coordinate (aBone) at (-2.10,-0.65);

			\coordinate (aLtwo) at (-1.20,0);
			\coordinate (aRtwo) at (-0.10,0);

			\coordinate (aLk) at (1.45,0);
			\coordinate (aRk) at (2.55,0);
			\coordinate (aTk) at (2.00,0.65);
			\coordinate (aBk) at (2.00,-0.65);

			\draw[side] (aLone)--(aTone)--(aRone);
			\draw[side] (aLone)--(aBone)--(aRone);
			\draw[central] (aTone)--(aBone);

			\draw (aLtwo) to[bend left=32] (aRtwo);
			\draw (aLtwo) to[bend right=32] (aRtwo);

			\draw (aLk)--(aTk)--(aRk);
			\draw (aLk)--(aBk)--(aRk);
			\draw (aTk)--(aBk);

			\draw (aRone)--(aLtwo);

			\coordinate (aMleft)  at (0.30,0);
			\coordinate (aMright) at (1.05,0);

			\draw (aRtwo)--(aMleft);
			\draw (aMright)--(aLk);

			\node[
				fill=white,
				inner sep=1.2pt
			] at (0.675,0) {$\cdots$};

			\foreach \p in {
				aLone,aRone,aTone,aBone,
				aLtwo,aRtwo,
				aLk,aRk,aTk,aBk}
				\node[point] at (\p) {};

			\node at (-2.10,-1.0) {$Q_1$};
			\node at (-0.65,-1.0) {$Q_2$};
			\node at (2.00,-1.0) {$Q_k$};

			\node at (0,-1.55) {(a)};

		\end{scope}

		\begin{scope}[shift={(3.4,0)}]

			\coordinate (bLone) at (-2.65,0);
			\coordinate (bRone) at (-1.55,0);
			\coordinate (bTone) at (-2.10,0.65);
			\coordinate (bBone) at (-2.10,-0.65);

			\coordinate (bLtwo) at (-1.20,0);
			\coordinate (bRtwo) at (-0.10,0);

			\coordinate (bLk) at (1.45,0);
			\coordinate (bRk) at (2.55,0);
			\coordinate (bTk) at (2.00,0.65);
			\coordinate (bBk) at (2.00,-0.65);

			\draw (bLone)--(bTone)--(bRone);
			\draw (bLone)--(bBone)--(bRone);
			\draw (bTone)--(bBone);

			\draw (bLtwo) to[bend left=32] (bRtwo);
			\draw (bLtwo) to[bend right=32] (bRtwo);

			\draw (bLk)--(bTk)--(bRk);
			\draw (bLk)--(bBk)--(bRk);
			\draw (bTk)--(bBk);

			\draw (bRone)--(bLtwo);

			\coordinate (bMleft)  at (0.30,0);
			\coordinate (bMright) at (1.05,0);

			\draw (bRtwo)--(bMleft);
			\draw (bMright)--(bLk);

			\node[
				fill=white,
				inner sep=1.2pt
			] at (0.675,0) {$\cdots$};

			\draw (bRk)
				to[out=70,in=110,looseness=0.9]
				(bLone);

			\foreach \p in {
				bLone,bRone,bTone,bBone,
				bLtwo,bRtwo,
				bLk,bRk,bTk,bBk}
				\node[point] at (\p) {};

			\node at (-2.10,-1.0) {$Q_1$};
			\node at (-0.65,-1.0) {$Q_2$};
			\node at (2.00,-1.0) {$Q_k$};

			\node at (0,-1.55) {(b)};

		\end{scope}

	\end{tikzpicture}

	\caption{(a) a D2C-string; (b) a D2C-ring.}
	\label{fig:diamonds}
\end{figure}

Let $H$ be a cubic graph and $u\in V(H)$. Let $v_1,v_2,v_3$ be the neighbours of $u$, which need not be distinct. Replacing $u$ with a triangle is to delete $u$, add a triangle $u_1u_2u_3$, and add edges  joining  $u_i$ and $v_i$ for $i\in\{1,2,3\}$.
 Replacing an edge $e$ of $H$ with a D2C-string $S_e$ is to delete $e$ and add two edges $e_1$ and $e_2$ joining the two ends of $e$ to the two vertices of degree two of $S_e$, respectively. The subgraph obtained from $S_e$ by adding $e_1$ and $e_2$  is called the \emph{channel} of $e$, and  $e_1$ and $e_2$  are called the \emph{pendant edges} of the channel.

\begin{lemma}[\cite{Mkrtchyan2025}]\label{thm:Oum}
A bridgeless claw-free cubic graph $G$ is one of the following:
\vspace{-6pt}
\begin{enumerate}[(i)]
	\setlength{\topsep}{0pt}
	\setlength{\partopsep}{0pt}
	\setlength{\itemsep}{-1pt}
	\setlength{\parsep}{0pt}
	\item $K_4$ or $K_2^3$, where $K_2^3$ consists of two vertices joined by three parallel edges;
	\item a D2C-ring containing at least two members;
	\item a graph obtained from a bridgeless cubic graph $H$ by replacing every vertex of $H$ by a triangle and replacing some edges of $H$ by D2C-strings.
	\end{enumerate}
\end{lemma}

The graphs $K_4$, $K_2^3$, and every D2C-ring are $3$-edge-colourable and hence can be covered by three perfect matchings. Thus, We need only consider graphs in \textup{(iii)}, referred to as graphs of type \textup{(iii)}.
Let $G$ be a graph of type (iii). We call $H$ the \emph{core} of $G$.
For each $e\in E(H)$, let $\ell(e)$ and $r(e)$ denote the numbers of diamonds and 2-cycles in its channel,  respectively. The edges of the channel of $e$ not contained in any diamond or 2-cycle are called the \emph{linking edges}.
If $\ell(e)=r(e)=0$, the channel of $e$ is an edge, which is a linking edge. Let 
$$m=|E(H)|, l=\sum_{e\in E(H)}\ell(e), r=\sum_{e\in E(H)}r(e).$$ Since $H$ is cubic, $|V(H)|=\frac{2}{3}m$. Since replacing each vertex of $H$ by a triangle gives $2m$ vertices, and the $l$ diamonds and $r$ 2-cycles add $4l+2r$ vertices, we have $|V(G)|=2m+4l+2r$. Since $G$ is cubic,
\begin{equation}\label{eq:size-G}
	|E(G)|=\frac{3}{2}|V(G)|=3m+6l+3r.
\end{equation}

Let $M$ be a perfect matching of $G$, and let $e$ be an edge of $H$. By the structure of the channel of $e$, either its two pendant edges belong to $M$ or neither does. We say that $M$ \emph{selects} $e$ if the two pendant edges of the channel of $e$ belong to $M$, or equivalently, if the two ends of $e$ are covered by the edges of $M$ in the channel of $e$.

A spanning subgraph $J$ of $G$ is a \emph{parity subgraph} if $d_J(v)\equiv d_G(v)\pmod 2$ for every $v\in V(G)$.
 Since $G$ is a cubic graph, we have $d_J(v)\in\{1,3\}$ for every $v\in V(G)$.
Recall that $G$ is a graph  of type \textup{(iii)} with core $H$. We say that a parity subgraph $J$ of $H$ can be \emph{extended} to a perfect matching $M$ of $G$ if the edges of $H$ selected by $M$ are precisely the edges of $J$.

\begin{lemma}\label{lem:lifting}
Let $G$ be a graph  of type \textup{(iii)} with core $H$. If $M$ is a perfect matching of $G$, then the  edges of $H$ selected by $M$ induce a parity subgraph of $H$.
Conversely, every parity subgraph of $H$ can be extended to a perfect matching of $G$.
\end{lemma}

\begin{proof}
Let $M$ be a perfect matching of $G$, and let $J$ be the spanning subgraph of $H$ induced by the edges of $H$ selected by $M$. For each $v\in V(H)$ and the triangle replacing $v$, either one or all three vertices are covered by edges of $M$ in the  channels containing them. Hence $d_J(v)\in\{1,3\}$. Since $H$ is a cubic graph, $J$ is a parity subgraph of $H$.

Conversely, let $J$ be a parity subgraph of $H$. We now extend $J$ to a perfect matching of $G$. For each edge $e$ of $H$, choose edges in its channel as follows.
When $e\in J$, if $\ell(e)=r(e)=0$, choose the only edge of its channel; otherwise, choose all linking edges and the central edge of every diamond. When $e\notin J$, choose a pair of opposite side edges from each diamond together with one edge from each 2-cycle. Let $M_1$ be the set of all edges chosen in this way.
Then $M_1$ is a matching of $G$, and for each channel $B$ and its two vertices of degree one, say $x$ and $y$, $M_1$ covers every  vertex of $B$ other than $x$ and $y$, and  $M_1$ covers  $x$ and $y$ if and only if $e\in J$.
For each $v\in V(H)$, consider the triangle replacing $v$.
If $d_J(v)=1$, exactly one vertex of the triangle is covered by $M_1$, and we choose the edge of the triangle not incident with this vertex.
If $d_J(v)=3$, all three vertices are covered by $M_1$, and we choose no edge of the triangle.
Let $M_2$ be the set of all edges chosen in this way.
Then $M_1\cup M_2$ is a perfect matching of $G$. This completes the proof.
\end{proof}

\section{Proof of Theorem~\ref{thm:five-cover}}\label{sec:five}

Let $\mathcal{F}$ denote the \emph{Fano plane}, which has seven points and seven lines, as shown in Figure~\ref{fig:fano}. We identify its point set $V(\mathcal{F})=\{001,010,011,100,101,110,111\}$ with the seven nonzero vectors of $\mathbb{F}_2^3$. Three distinct points $x,y,z\in V(\mathcal{F})$ lie on a common line if and only if $x+y+z=000$. For $\mathcal{F}$, any two distinct lines meet in exactly one point and every point lies on exactly three lines.
A \emph{Fano colouring} of a cubic graph $G$ is a function $\tau:E(G)\to V(\mathcal{F})$ such that, at each vertex $v$ of $G$, the colours of the three edges incident with $v$ form a line of $\mathcal{F}$.
\begin{figure}[ht]
	\centering
	\begin{tikzpicture}[
		scale=1.0,
		line width=0.8pt,
		point/.style={circle,draw,fill=white,inner sep=1.8pt}
		]
		
		\coordinate (A) at (0,2.1);
		\coordinate (B) at (-0.9,0.55);
		\coordinate (C) at (0.9,0.55);
		\coordinate (D) at (-1.8,-1.05);
		\coordinate (E) at (0,-1.05);
		\coordinate (F) at (1.8,-1.05);
		\coordinate (G) at (0,0);
		
		\draw (A)--(D);
		\draw (A)--(F);
		\draw (D)--(F);
		\draw (A)--(E);
		\draw (B)--(F);
		\draw (D)--(C);
		\draw (0,-0.02) circle[radius=1.06];
		
		\node[point,label=above:$001$] at (A) {};
		\node[point,label=left:$101$] at (B) {};
		\node[point,label=right:$011$] at (C) {};
		\node[point,label=left:$100$] at (D) {};
		\node[point,label=below:$110$] at (E) {};
		\node[point,label=right:$010$] at (F) {};
		\node[point,label=right:$111$] at (G) {};
		
	\end{tikzpicture}
	\caption{The Fano plane $\mathcal{F}$.}
	\label{fig:fano}
\end{figure}
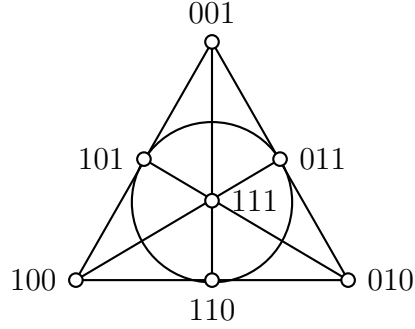

\begin{lemma}[\cite{MacajovaSkoviera2005}]\label{lem:sixline}
	Every bridgeless cubic graph has a Fano colouring that uses at most six lines of the Fano plane.
\end{lemma}

\begin{lemma}\label{lem:five-parity}
Every bridgeless cubic graph $H$ has five parity subgraphs $J_1,\ldots,J_5$ satisfying the following properties:
\vspace{-6pt}
\begin{enumerate}[(i)]
	\setlength{\topsep}{0pt}
	\setlength{\partopsep}{0pt}
	\setlength{\itemsep}{-1pt}
	\setlength{\parsep}{0pt}
\item every edge of $H$ belongs to at least one $J_i$;
\item for every edge $e\in E(H)$, at least two of $J_1,\ldots,J_5$ do not contain $e$;
\item for every vertex $v\in V(H)$ and every edge $e$ incident with $v$, there is a $J_i$ that contains $e$ but no other edge incident with $v$.
\end{enumerate}
\end{lemma}

\begin{proof}
By Lemma~\ref{lem:sixline}, there exists a Fano colouring $\tau$ of $H$ that uses at most six lines.
For each line $L$ of $\mathcal{F}$, let $J_L$ be the spanning subgraph of $H$ with
$E(J_L)=\{e\in E(H):\tau(e)\in L\}$.
We first show that $J_L$ is a parity subgraph of $H$.
For each $v\in V(H)$, let $L_v$ be the line formed by the colours of the three edges incident with $v$.
If $L\neq L_v$, then $L$ and $L_v$ meet in exactly one point, and hence $d_{J_L}(v)=1$; if $L=L_v$, then $d_{J_L}(v)=3$.
Thus, $J_L$ is a parity subgraph of $H$ for every line $L$ of $\mathcal{F}$.

Since $\tau$ uses at most six lines, there is a line $L_0$ not used by $\tau$.
Let $x\in L_0$, and let $L_1$ and $L_2$ be the other two lines containing $x$.
The five lines other than $L_1$ and $L_2$ determine five parity subgraphs of $H$, not necessarily distinct, which we denote by $J_1,\ldots,J_5$.
Since each point of $\mathcal{F}$ lies on exactly three lines, every edge of $H$ is contained in exactly three of the seven parity subgraphs $J_L$ and is not contained in the other four.
After omitting $J_{L_1}$ and $J_{L_2}$, at least one of the remaining five parity subgraphs contains the edge, while at least two of them do not.
Hence, every edge of $H$ is contained in at least one of $J_1,\ldots,J_5$ and is not contained in at least two of them.
This proves \textup{(i)} and \textup{(ii)}.

For \textup{(iii)}, let $v\in V(H)$, $e$ be an edge incident with $v$, and $y=\tau(e)$. Let $L_1'$ and $L_2'$ be the two lines other than $L_v$ containing $y$. Since $L_i'\cap L_v=\{y\}$ for $i\in\{1,2\}$, $J_{L_i'}$ contains $e$ and no other edge incident with $v$. Suppose that neither $J_{L_1'}$ nor $J_{L_2'}$ is one of $J_1,\ldots,J_5$. Then $\{L_1',L_2'\}=\{L_1,L_2\}$, and hence $\{y\}=L_1'\cap L_2'=L_1\cap L_2=\{x\}$. Thus $L_v=L_0$, contradicting that $L_0$ is not used by $\tau$. So one of $J_{L_1'}$ and $J_{L_2'}$ satisfies \textup{(iii)}. This completes the proof.
\end{proof}

\begin{proof}[Proof of Theorem~\ref{thm:five-cover}]
Let $G$ be a graph of type \textup{(iii)} with core $H$.  Let $J_1,\ldots,J_5$ be parity subgraphs of $H$ satisfying conditions \textup{(i)}-\textup{(iii)} of Lemma~\ref{lem:five-parity}. By Lemma~\ref{lem:lifting}, for each $i\in\{1,\ldots,5\}$, choose a perfect matching $M_i$ of $G$ extending $J_i$.

Let $e$ be an edge of $H$ and $B_e$ the channel of $e$. By Lemma~\ref{lem:five-parity} \textup{(i)} and \textup{(ii)}, there exist distinct indices $i,j,k$ in $\{1,\ldots,5\}$ such that $J_i$ contains $e$, while $J_j$ and $J_k$ do not.
Thus, $M_i$ selects $e$ and contains all linking edges in $B_e$ and the central edge of each diamond in $B_e$.
Moreover, $M_j$ and $M_k$ do not select $e$. In each diamond of $B_e$, $M_j$ contains one pair of opposite side edges and $M_k$ contains the other pair; in each 2-cycle, $M_j$ and $M_k$ contain different edges.
 Hence $\{M_i, M_j, M_k\}$ covers the channel of $e$.

 Let $v\in V(H)$ and let $f$ be an edge incident with $v$. By Lemma~\ref{lem:five-parity}\textup{(iii)}, there is an $s$ in $\{1,\ldots,5\}$ such that  $J_s$ contains $f$ but no other edge incident with $v$. Thus, $d_{J_s}(v)=1$.
 By the construction of $M_s$, in the triangle replacing $v$, the matching $M_s$ contains the edge not incident with  the vertex joined to the channel of $f$.
  For the three edges incident with $v$, the vertices joined to their channels are precisely the three vertices of the triangle replacing $v$. Hence  the triangle is covered by $\{M_1,\ldots,M_5\}$.
Since $v$ was arbitrary and  every channel is also covered, $\{M_1,\ldots,M_5\}$ cover $G$, this completes the proof.
\end{proof}

\section{Proof of Theorem~\ref{thm:intro-main}}
A graph is called \emph{matchable} if it has a perfect matching.  Let $H$ be a matchable graph.
For $X\subseteq V(H)$, let $\partial(X)$ denote the set of edges with exactly one end in $X$, and let $\partial(v)=\partial(\{v\})$ for $v\in V(H)$. A cut $\partial(X)$ is \emph{odd} if $|X|$ is odd. For $x\in\mathbb{R}^{E(H)}$ and $C\subseteq E(H)$, let $x(C)=\sum_{e\in C}x(e)$.
For a perfect matching $M$ of $H$, its \emph{characteristic vector} $\chi^M$ is defined by $\chi^M(e)=1$ if $e\in M$ and $\chi^M(e)=0$ otherwise. The \emph{perfect matching polytope} of $H$, denoted by $\mathcal{P}(H)$, is the convex hull of the characteristic vectors of all perfect matchings of $H$.
Edmonds' perfect matching polytope theorem~\cite{Edmonds1965} states that a vector $x\in\mathcal{P}(H)$  if and only if
\[
\begin{cases}
x(e)\geq 0, & \text{for every } e\in E(H),\\
x(\partial(v))=1, & \text{for every } v\in V(H),\\
x(C)\geq 1, & \text{for every odd cut } C \text{ of } H.
\end{cases}
\]
A vector in $\mathcal{P}(H)$ is called a \emph{fractional perfect matching} of $H$. Thus every fractional perfect matching $x$ of $H$ can be written as a convex combination of the characteristic vectors of perfect matchings of $H$, that is, there exist perfect matchings $M_1,\ldots,M_k$ of $H$ and nonnegative numbers $\lambda_1,\ldots,\lambda_k$ such that $\sum_{i=1}^k\lambda_i=1$ and $x=\sum_{i=1}^k\lambda_i\chi^{M_i}$.

\begin{lemma}\label{lem:selection}
Let $x$ be a fractional perfect matching of a matchable graph $H$, and let $c:E(H)\to\mathbb R$ be an edge-weight function.
Then there exists a perfect matching $M$ of $H$ such that
\[
c(M)\leq\sum_{e\in E(H)}c(e)x(e)
\]
and $|M\cap C|=1$ for every odd cut $C$ with $x(C)=1$.
\end{lemma}

\begin{proof}
By the perfect matching polytope theorem, write $x=\sum_{i=1}^k\lambda_i\chi^{M_i}$ as above. Then
\[
\begin{aligned}
\sum_{e\in E(H)}c(e)x(e)
&=\sum_{e\in E(H)}c(e)\sum_{i=1}^k\lambda_i\chi^{M_i}(e)
&=\sum_{i=1}^k\lambda_i\sum_{e\in E(H)}c(e)\chi^{M_i}(e)
=\sum_{i=1}^k\lambda_i c(M_i).
\end{aligned}
\]
Choose $j$ such that $\lambda_j>0$ and $c(M_j)$ is minimum among all $c(M_i)$ with $\lambda_i>0$. Then
\[
c(M_j)=\sum_{i=1}^k\lambda_i c(M_j)
\leq\sum_{i=1}^k\lambda_i c(M_i)
=\sum_{e\in E(H)}c(e)x(e).
\]
Let $C$ be an odd cut with $x(C)=1$. Then  for each $i$,  $|M_i\cap C|$ is odd, and so $|M_i\cap C|\geq1$. Moreover,
\[
\begin{aligned}
1=x(C)
&=\sum_{e\in C}x(e)
=\sum_{e\in C}\sum_{i=1}^k\lambda_i\chi^{M_i}(e)
&=\sum_{i=1}^k\lambda_i\sum_{e\in C}\chi^{M_i}(e)
=\sum_{i=1}^k\lambda_i|M_i\cap C|.
\end{aligned}
\]
Since  $\sum_{i=1}^k\lambda_i=1$, if $\lambda_i>0$, then $|M_i\cap C|=1$, in particular,  $|M_j\cap C|=1$. Taking $M=M_j$ completes the proof.
\end{proof}

The first bound in the following lemma was established by DeVos et al.~\cite{DeVosEtAl2020}. For completeness, we include its proof here.

\begin{lemma}\label{lem:three-matchings}
Let $H$ be a bridgeless cubic graph and let $w:E(H)\to\mathbb R$ be a nonnegative edge-weight function.
Then there exist perfect matchings $M_1,M_2,M_3$ of $H$ such that
\[
w(M_1\cap M_2)\leq\frac{1}{15}w(E(H)) \text{ and } w(M_1\cap M_2\cap M_3)\leq\frac{1}{105}w(E(H)).
\]
\end{lemma}

\begin{proof}
Let $x(e)=\frac{1}{3}$ for every $e\in E(H)$. Since $H$ is cubic, $x(\partial(v))=1$ for every $v\in V(H)$, and $|\partial(X)|$ has the same parity as $|X|$ for every $X\subseteq V(H)$. Since $H$ is bridgeless, every odd cut has at least three edges. Hence $x(C)\geq1$ for every odd cut $C$, and therefore $x$ is a fractional perfect matching of $H$.
Note that every $3$-edge-cut $C$ of $H$ is an odd cut and satisfies $x(C)=1$.
By Lemma~\ref{lem:selection}, there exists a perfect matching $M_1$ of $H$ such that $w(M_1)\leq\sum_{e\in E(H)}w(e)x(e)=\frac{1}{3}w(E(H))$ and $|M_1\cap C|=1$ for every $3$-edge-cut $C$ of $H$.

Define
\[
y(e)=
\begin{cases}
	\frac{1}{5},&e\in M_1,\\
	\frac{2}{5},&e\notin M_1.
\end{cases}
\]
Since $M_1$ is a perfect matching of $H$, each $v\in V(H)$ is incident with exactly one edge of $M_1$, and hence $y(\partial(v))=1$.
For an odd cut $C$ of $H$, if $|C|=3$, then $|M_1\cap C|=1$, and hence $y(C)=\frac{1}{5}+2\cdot\frac{2}{5}=1$; if $|C|\geq5$, then $y(C)\geq\frac{1}{5}|C|\geq1$. Therefore, $y$ is a fractional perfect matching of $H$.

Define
\[
c_1(e)=
\begin{cases}
	w(e),&e\in M_1,\\
	0,&e\notin M_1.
\end{cases}
\]
By Lemma~\ref{lem:selection}, there exists a perfect matching $M_2$ of $H$ such that
\[
\begin{aligned}
w(M_1\cap M_2)
=c_1(M_2)
&\leq \sum_{e\in E(H)}c_1(e)y(e)
&=\sum_{e\in M_1}w(e)y(e)
=\frac{1}{5}w(M_1)
\leq\frac{1}{15}w(E(H)).
\end{aligned}
\]
For every $3$-edge-cut $C$, since $y(C)=1$, we have $|M_2\cap C|=1$. Moreover, if $C$ is a $5$-edge-cut with $C\subseteq M_1$, then $y(C)=1$, and hence $|M_2\cap C|=1$.

Define
\[
z(e)=
\begin{cases}
	\frac{1}{7},&e\in M_1\cap M_2,\\
	\frac{2}{7},&e\in M_1\triangle M_2,\\
	\frac{3}{7},&e\notin M_1\cup M_2.
\end{cases}
\]
Since $M_1$ and $M_2$ are perfect matchings of $H$, for each $v\in V(H)$, the edges of $M_1$ and $M_2$ incident with $v$ are either the same or distinct. If they are the same, then $z(\partial(v))=\frac{1}{7}+\frac{3}{7}+\frac{3}{7}=1$; if they are distinct, then $z(\partial(v))=\frac{2}{7}+\frac{2}{7}+\frac{3}{7}=1$. Thus $z(\partial(v))=1$ for every $v\in V(H)$.
For an odd cut $C$ of $H$,
\[
z(C)=\sum_{e\in C}\left(\frac{3}{7}-\frac{1}{7}\chi^{M_1}(e)-\frac{1}{7}\chi^{M_2}(e)\right)
=\frac{3|C|-|M_1\cap C|-|M_2\cap C|}{7}.
\]
If $|C|=3$, then $|M_1\cap C|=|M_2\cap C|=1$, and hence $z(C)=1$. Suppose that $|C|=5$, and let $p_1=|M_1\cap C|$ and $p_2=|M_2\cap C|$. Then $z(C)=\frac{15-p_1-p_2}{7}$. Since $C$ is an odd cut, both $p_1$ and $p_2$ are odd, and hence belong to $\{1,3,5\}$. If $p_1\leq3$, then $p_1+p_2\leq3+5=8$. If $p_1=5$, then $C\subseteq M_1$, so $y(C)=1$, and Lemma~\ref{lem:selection} gives $p_2=1$. Thus $p_1+p_2\leq8$, and hence $z(C)\geq1$. If $|C|\geq7$, then $z(C)\geq\frac{1}{7}|C|\geq1$. Therefore, $z$ is a fractional perfect matching.

Define
\[
c_2(e)=
\begin{cases}
	w(e),&e\in M_1\cap M_2,\\
	0,&e\notin M_1\cap M_2.
\end{cases}
\]
By Lemma~\ref{lem:selection}, there exists a perfect matching $M_3$ of $H$ such that
\[
\begin{aligned}
w(M_1\cap M_2\cap M_3)
&=c_2(M_3)
\leq\sum_{e\in E(H)}c_2(e)z(e)\\
&=\sum_{e\in M_1\cap M_2}w(e)z(e)
=\frac{1}{7}w(M_1\cap M_2)
\leq\frac{1}{105}w(E(H)).
\end{aligned}
\]
This completes the proof.
\end{proof}

In the rest of this section, let $G$ be a graph of type \textup{(iii)} with core $H$. Then  $H$ is a bridgeless cubic graph. Let $w:E(H)\to\mathbb R$ be a nonnegative edge-weight function. By Lemma~\ref{lem:three-matchings}, $H$ has three perfect matchings $M_1,M_2,M_3$  satisfying $w(M_1\cap M_2)\leq\frac{1}{15}w(E(H)) \text{ and } w(M_1\cap M_2\cap M_3)\leq\frac{1}{105}w(E(H))$. Let $B=M_1\cap M_2$, and define spanning subgraphs $J_1,J_2,J_3,J_4$ of $H$ by $E(J_1)=M_1$, $E(J_2)=M_2$, $E(J_3)=M_3$, and $E(J_4)=E(H)\setminus(M_1\triangle M_2)$. Then $B$ is a matching of $H$, $w(B)\leq\frac{1}{15}w(E(H))$ and $B\subseteq E(J_1)\cap E(J_2)\cap E(J_4)$.

\begin{lemma}\label{lem:core-construction}
The subgraphs $J_1,J_2,J_3$, and $J_4$ are parity subgraphs of $H$, and together with $B$ satisfy the following properties: 
\vspace{-6pt}
\begin{enumerate}[(i)]
	\setlength{\topsep}{0pt}
	\setlength{\partopsep}{0pt}
	\setlength{\itemsep}{-1pt}
	\setlength{\parsep}{0pt}
\item $w\bigl(E(J_1)\cap E(J_2)\cap E(J_3)\cap E(J_4)\bigr)\leq\frac{1}{105}w(E(H))$;
\item every edge of $B$ is contained in at least three of $J_1,J_2,J_3,J_4$;
\item every edge in $E(H)\setminus B$ is contained in one or two of $J_1,J_2,J_3,J_4$.
\end{enumerate}
\end{lemma}

\begin{proof}
For $i\in\{1,2,3\}$, since $E(J_i)=M_i$,  $J_i$ is a parity subgraph of $H$. In the spanning subgraph of $H$ with edge set $M_1\triangle M_2$, every vertex has degree zero or two. Since $H$ is cubic, every vertex of $J_4$ has degree one or three and hence $J_4$ is also a parity subgraph of $H$.

Since $E(J_1)\cap E(J_2)\cap E(J_3)\cap E(J_4)=M_1\cap M_2\cap M_3$, by Lemma~\ref{lem:three-matchings}, \textup{(i)} follows. Since every edge $e\in B$ is contained in $J_1$, $J_2$, and $J_4$, \textup{(ii)} follows.

Now, let $e\in E(H)\setminus B$. Then $e\notin M_1\cap M_2$. If $e\in M_1\cup M_2$, then $e$ is contained in exactly one of $J_1$ and $J_2$ and is not contained in $J_4$. If $e\notin M_1\cup M_2$, then $e$ is contained in $J_4$ and neither in $J_1$ nor in $J_2$. Thus $e$ is contained in exactly one of $J_1,J_2,J_4$ and possibly  in $J_3$. Hence $e$ is contained in one or two of $J_1,J_2,J_3,J_4$.  \textup{(iii)} follows.
\end{proof}

In the following, when $J_1,J_2,J_3$ and $J_4$ are extended to perfect matchings of $G$, the edges in each channel and in each triangle replacing a vertex of $H$ are chosen as in the proof of Theorem~\ref{thm:five-cover}. Recall that  $m=|E(H)|$, $l=\sum_{e\in E(H)}\ell(e)$, and $r=\sum_{e\in E(H)}r(e)$.

\begin{lemma}\label{lem:local-loss}
There exist four perfect matchings of $G$ such that at most $\frac{16}{105}\left(m+l+\frac{r}{2}\right)$ edges of $G$ are uncovered.
\end{lemma}

\begin{proof}
By Lemma~\ref{lem:lifting}, there exist perfect matchings $F_1,F_2,F_3,F_4$ of $G$ such that $J_i$ is extended to $F_i$ for each $i\in\{1,2,3,4\}$. 
Let $\mathbb{F}=\{F_1,F_2,F_3,F_4\}$. We first show that $\mathbb{F}$ satisfies the following properties:
\vspace{-6pt}
\begin{enumerate}[(i)]
	\setlength{\topsep}{0pt}
	\setlength{\partopsep}{0pt}
	\setlength{\itemsep}{-1pt}
	\setlength{\parsep}{0pt}
\item for every $e\in E(H)\setminus B$,  the channel of $e$ is covered by $\mathbb{F}$;
\item for every $v\in V(H)$ not incident with any edge of $B$, the triangle replacing $v$ is covered by $\mathbb{F}$;
\item for every $e=uv\in B$ contained in exactly three of $J_1,\ldots,J_4$, exactly $2\ell(e)+r(e)+2$ edges in the channel of $e$ and in the two triangles replacing $u$ and $v$ are uncovered by $\mathbb{F}$;
\item for every $e=uv\in B$ contained in each of $J_1,\ldots,J_4$, exactly $4\ell(e)+2r(e)+4$ edges in the channel of $e$ and in the two triangles replacing $u$ and $v$ are uncovered by $\mathbb{F}$.
\end{enumerate}

Let $e\in E(H)\setminus B$. By Lemma~\ref{lem:core-construction}\textup{(iii)}, $e$ is contained in one or two of $J_1,J_2,J_3,J_4$. Hence at least one member of $\mathbb{F}$ selects $e$, while at least two do not. Therefore  the channel of $e$ is covered by $\mathbb{F}$. (i) follows. 

Let $v\in V(H)$ be a vertex not incident with any edge of $B$. Since $B=M_1\cap M_2$, the edges of $M_1$ and $M_2$ incident with $v$ are distinct. The third edge incident with $v$ is contained in neither $M_1$ nor $M_2$, and hence is contained in $J_4$. Thus the unique edges of $J_1,J_2,J_4$ incident with $v$ are pairwise distinct. Consequently, $\{F_1,F_2,F_4\}$ covers the triangle replacing $v$. (ii) follows. 

Now let $e=uv\in B$. Then $e$ is contained in $J_1$ and $J_2$, and $e$ is not contained in $M_1\triangle M_2$. Since $M_1$ and $M_2$ are perfect matchings, the other two edges incident with $u$ (or with $v$) are contained in neither $M_1$ nor $M_2$, and hence are contained in $E(J_4)$.
 
Suppose that $e$ is contained in exactly three of $J_1,J_2,J_3,J_4$. Since $e$ is contained in $J_1,J_2,J_4$, it is not contained in $J_3$. Hence, for any  end of $e$, $J_3$ contains one edge incident with it other than $e$.
In each of the triangles replacing $u$ and $v$, $F_1$ and $F_2$ contain the same edge of it, $F_3$ contains one of the other two edges of it, and $F_4$ contains no edge of it. Hence exactly one edge of each triangle is uncovered by  $\mathbb{F}$. Since $F_3$ is the only perfect matching not selecting $e$, exactly two side edges of each diamond and one edge of each 2-cycle in the channel of $e$ are uncovered. Thus the total number of edges uncovered by  $\mathbb{F}$ is $2\ell(e)+r(e)+2$. \textup{(iii)}  follows. 

Suppose that $e$ is contained in all of $J_1,J_2,J_3,J_4$. Then the three edges incident with $u$ (or with $v$) belong to $E(J_4)$. In each triangle replacing $u$ and $v$, $F_1,F_2,F_3$ contain a common edge of the triangle,
 while $F_4$ contains no edge of the triangle. Hence exactly two edges of the triangle are uncovered by $\mathbb{F}$. Since all four perfect matchings select $e$, all four side edges of each diamond and both edges of each 2-cycle in the channel of $e$ are uncovered. Thus, the  number of edges uncovered by $\mathbb{F}$ is $4\ell(e)+2r(e)+4$.  \textup{(iv)} follows.
 
Let $w:E(H)\to\mathbb R$ be defined by $w(e)=\ell(e)+\frac{r(e)}{2}+1$ for every $e\in E(H)$. Since $\ell(e),r(e)\geq0$, the function $w$ is nonnegative and $w(E(H))=m+l+\frac{r}{2}$.
We now complete the proof by applying the four properties above to the family $\mathbb{F}$ of perfect matchings constructed from $J_1,J_2,J_3,J_4$, where $J_1,J_2,J_3,J_4$ are obtained with respect to this function $w$.

Let $B'=E(J_1)\cap E(J_2)\cap E(J_3)\cap E(J_4)$. Then $B'\subseteq B$, and $w(B)\leq \frac{1}{15}w(E(H)) =\frac{1}{15}\left(m+l+\frac{r}{2}\right)$.
By Lemma~\ref{lem:core-construction}\textup{(i)}, we have  $w(B')\leq\frac{1}{105}\left(m+l+\frac{r}{2}\right)$. By Lemma~\ref{lem:core-construction}\textup{(ii)}, every edge in $B\setminus B'$ is contained in exactly three of $J_1,J_2,J_3,J_4$.

By property \textup{(i)} and \textup{(ii)} of $\mathbb{F}$, every edge not covered by $\mathbb{F}$ lies either in the channel of an edge of $B$ or in a triangle replacing an end of an edge of $B$. Note that the channels of distinct edges of $H$ are edge-disjoint. Since $B$ is a matching, the triangles replacing the ends of distinct edges of $B$ are distinct. 
By properties \textup{(iii)} and \textup{(iv)} of $\mathbb{F}$, each edge $e\in B\setminus B'$ contributes $2w(e)$ uncovered edges, while each edge $e\in B'$ contributes $4w(e)$ uncovered edges. Thus the total number of edges  not covered by $\mathbb{F}$ is
\[
\begin{aligned}
2w(B\setminus B')+4w(B')
&=2w(B)+2w(B')\\
&\leq 2\cdot\frac{1}{15}\left(m+l+\frac{r}{2}\right)
+2\cdot\frac{1}{105}\left(m+l+\frac{r}{2}\right)\\
&=\frac{16}{105}\left(m+l+\frac{r}{2}\right).
\end{aligned}
\]
This completes the proof.
\end{proof}

\begin{proof}[Proof of Theorem~\ref{thm:intro-main}]
The graphs $K_4$, $K_2^3$, and every D2C-ring can be covered by three perfect matchings and hence satisfy the theorem. For a graph $G$ of type \textup{(iii)}, by Lemma~\ref{lem:local-loss}, there exist four perfect matchings of $G$ such that at most $\frac{16}{105}\left(m+l+\frac{r}{2}\right)$ edges are uncovered by them.
By~\eqref{eq:size-G}, $|E(G)|=3m+6l+3r$. Since $l,r\geq0$, we have
\[
\frac{\frac{16}{105}\left(m+l+\frac{r}{2}\right)}{|E(G)|}
=\frac{16\left(m+l+\frac{r}{2}\right)}{315(m+2l+r)}
\leq\frac{16}{315}.
\]
Thus at least $\frac{299}{315}|E(G)|$ edges are covered by these four perfect matchings.
\end{proof}

\section*{Acknowledgements}
This work is supported by the National Natural Science Foundation of China (Nos. 12571381
and 12371361) and Natural Science Foundation of Henan Province (Nos. 262300422603 and 262300421874).

\end{document}